\documentclass[12pt]{article}
\usepackage{graphics,latexsym,amssymb,amsmath,amsthm,bm}
\usepackage[dvips]{graphicx}
\input{amssym.def} 
\usepackage{psfrag}
\usepackage{color}

\newtheorem{thm}{Theorem}[section]
\newtheorem{dfn}[thm]{Definition}

\newtheorem{prop}[thm]{Proposition}
\newtheorem{lm}[thm]{Lemma}
\newtheorem{cor}[thm]{Corollary}
\newtheorem{remark}[thm]{Remark}

\theoremstyle{definition}

\newcommand{\red}{\textcolor{red}}
\newcommand{\green}{\textcolor{green}}
\newcommand{\blue}{\textcolor{blue}}

\title{Minimizing length of billiard trajectories in hyperbolic polygons}

\author{John R. Parker, Norbert Peyerimhoff, and Karl Friedrich Siburg}

\begin{document}

\maketitle

\begin{abstract}
  Closed billiard trajectories in a polygon in the hyperbolic plane
  can be coded by the order in which they hit the sides of the
  polygon. In this paper, we consider the average length of cyclically
  related closed billiard trajectories in ideal hyperbolic polygons
  and prove the conjecture that this average length is minimized for
  regular hyperbolic polygons. The proof uses a strict convexity
  property of the geodesic length function in Teichm\"uller space with
  respect to the Weil-Petersson metric, a fundamental result
  established by Wolpert.
\end{abstract}

\tableofcontents

\section{Introduction}

To play billiards in a Euclidean polygon, the rules are as follows: An
infinitesimal ball travels along a straight line (geodesic) at
constant speed, and when it hits a side of the polygon then it changes
its direction so the angle of incidence agrees with the angle of
reflection. The path followed by such a ball is called a {\em billiard
  trajectory}.

It also makes sense to play billiards in a hyperbolic polygon, as here
we also have well-defined meanings of geodesics and angles of
incidence and reflection. To our knowledge, the first instance where
such a dynamical system was considered is in an article by E. Artin
\cite{Art} written in German (see \cite{Dou} for an English
translation). Using continued fractions, he constructs dense
bi-infinite billiard trajectories in half of the fundamental polygon
of the modular surface. In fact, there are many striking connections
between geodesics on the modular surface, their symbolic coding via
cutting sequences and number theory such as binary quadratic forms and
continued fractions (see \cite{Ser} for a well known classical
reference and also, e.g., \cite{BourKont} for very recent
developments).

\begin{figure}[h]
\begin{center}
 \psfrag{DD}{${\mathbb D}^2$}
 \psfrag{P}{${\bf P}$}
 \includegraphics[width=11cm]{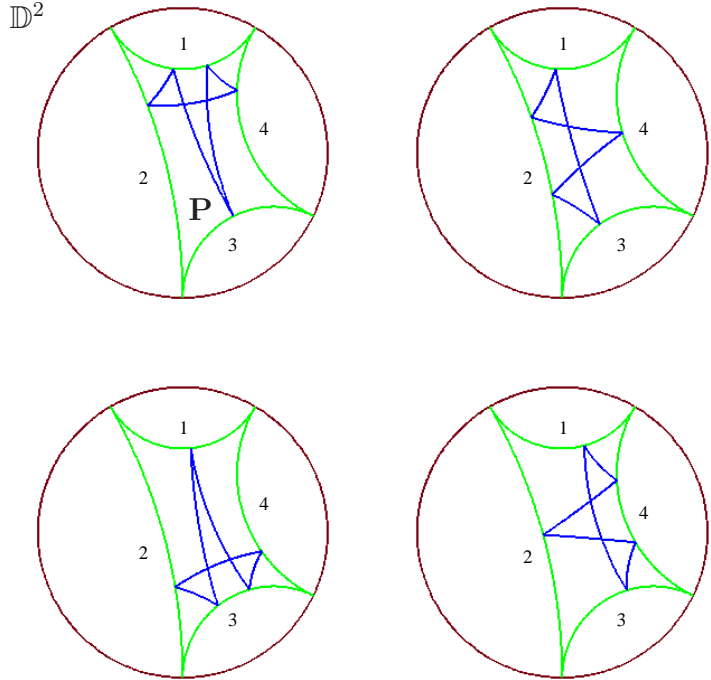}
 \caption{Illustration of all closed billiard trajectories in an ideal
   quadrilateral in the Poincar{\'e} disk ${\mathbb D}^2$ which are
   cyclically related to $(1,2,4,1,3)$. Their billiard
   sequences are $(2,3,1,2,4)$, $(3,4,2,3,1)$ and
   $(4,1,3,4,2)$. \label{fig:cycrelillus}}
\end{center}
\end{figure}

A billiard trajectory is said to be {\em closed} if after a finite
time it returns to the same point with the same direction. A natural
setting is to consider closed billiard trajectories in ideal
hyperbolic polygons where all vertices lie on the boundary at
infinity. A key piece of information of a billiard trajectory in such
an ideal hyperbolic polygon is its {\em billiard sequence} obtained by
recording the order of the sides where the ball hits the
boundary. Unlike the Euclidean case, where there may be uncountably
many closed billiard trajectories, although they are homotopic, with
the same billiard sequence, in the hyperbolic case there is at most
one (which is a consequence of the fact that the curvature is strictly
negative). In a polygon with $k$-sides, there are $k$ different {\em
  anti-clockwise enumerations} of the polygon's sides with labels
$1,2,\dots,k$. For each such labelling of the polygon, a closed
billiard trajectory $\gamma$ gives rise to a finite sequence of these
numbers, which we call the billiard sequence of $\gamma$ with respect
to this labelling. Two billiard trajectories in a given ideal polygon
are said to be {\em cyclically related} if, under different
anti-clockwise enumerations of the polygon's sides, they yield the
same billiard sequences. Figure \ref{fig:cycrelillus} illustrates this
concept: Our hyperbolic polygon there is an ideal quadrilateral in the
Poincar{\'e} unit disc ${\mathbb D}^2$ and the original closed
billiard trajectory has the finite billiard sequence $(1,2,4,1,3)$.

A closed billiard trajectory has a well-defined {\em hyperbolic
  length}. Given two different ideal hyperbolic $k$-gons with
anti-clockwise labellings of their sides, we can compare billiard
trajectories in both polygons with the same billiard sequences. For
closed billiard trajectories with a given billiard sequence, it is
interesting to ask for which polygons this length is minimised. A
first conjecture may be that the minimising polygon is the regular
polygon, i.e., the polygon whose symmetry group is the full dihedral
group. However, it is easy to see that this is actually not the
case. Indeed, there are many billiard sequences for which we can find
polygons whose corresponding trajectories have arbitrary small
lengths. Therefore, we consider {\em families of all cyclically
  related} closed billiard trajectories in a polygon and their
averaged lengths. A conjecture in \cite{CPS} states for {\em ideal}
hyperbolic polygons that the {\em average length function} of any
family of cyclically related closed billiard trajectories is uniquely
minimised by the regular polygon. Note that in a regular polygon,
cyclically related billiard trajectories are just rotations of each
other about the centre of the polygon and that all of them have the
same length. The aim of this paper is to confirm this conjecture, that is,
to prove the following result:

\begin{thm} \label{thm:main} In anti-clockwise labelled ideal
  hyperbolic polygons with $k \ge 3$ sides, the {\em average length
    function} of any family of cyclically related closed billiard
  trajectories corresponding to a given billiard sequence is uniquely
  minimised by the regular polygon.
\end{thm}

Let us roughly outline the strategy of proof: Firstly, we associate to
every polygon a hyperbolic surface by gluing two oppositely oriented
copies of the polygon pointwise along corresponding sides. Note that
the surface is noncompact and has $k$ cusps. We refer to this surface
as a {\em billiard surface}. Then every even-sided closed billiard
trajectory in the polygon lifts to a pair of closed geodesics of the
same length in the corresponding surface and every odd-sided closed
billiard trajectory lifts to one closed geodesic in the surface of
twice the length of the original billiard trajectory. In short,
billiard trajectories in the polygon correspond to geodesics in the
billiard surface.

This allows us to rephrase the conjecture as a result on the length of
closed geodesics in billiard surfaces. In order to apply powerful
results of Teichm\"uller theory and the {\em Weil-Petersson metric},
we consider the billiard surfaces as points in {\em Teichm\"uller
  space} and we call the subspace of all billiard surfaces the {\em
  billiard space}. Specifically, we use the results of Wolpert that
geodesic length functions are strictly convex with respect to the
Weil-Petersson Metric of Teichm\"uller space, and of Kerckhoff that
summed length functions of filling curves are proper.

Introducing a Weil-Petersson isometry of Teichm\"uller space (the so
called {\em flip map}) which fixes the billiard space pointwise, we
show that the billiard space is a geodesically convex subset of
Teichm\"uller space with respect to the Weil-Petersson metric. The
average length function of a family of cyclically related closed
billiard trajectories corresponds to a geodesic length function of a filling
family of closed curves in Teichm\"uller space. By the above mentioned
results, the geodesic length function becomes minimal in a unique
point in Teichm\"uller space and it only remains to show that this
minimising point in Teichm\"uller space corresponds to the billiard
surface associated to the regular polygon.

The following five sections of this paper follow essentially the
arguments of proof described above. In Appendix A, we briefly discuss
an analogous problem in the Euclidean setting: Here the polygons are
rectangles of area one and the unique minimizing billiard table turns
out to be the unit square.

\medskip

{\bf Acknowledgement:} The authors gratefully acknowledge the
inspiring and helpful discussions with Andreas Knauf and Joan Porti
concerning the strategy of proof. They also thank Andy Hayden for
numerous general detailled discussions concerning the topic and many
aspects of the proof. NP greatly enjoyed the hospitality of the TU
Dortmund and the Isaac Newton Institute, Cambridge, while he was
working on certain parts of this article.

\section{Cyclically related billiard trajectories are filling}

In this section, we prove a particular property of the connected
components of the complement of a union of all rotations of a closed
billiard trajectory in a regular hyperbolic polygon. Let us first
introduce this property for families of curves, which is called {\em
  filling}, in full detail. This definition for polygons is guided by
the desire that the lift of a family of filling curves in the
corresponding billiard surface (which will be introduced later) is
also filling. Note that, for finite volume Riemann surfaces, a family
of curves is called filling if each connected component of their
complement is an open topological disc or a disc with one puncture
corresponding to one of the cusps of the surface. Since the main
result of this section (Proposition \ref{prop-fill} below) holds both
for compact and ideal regular polygons, we formulate it in this
generality. Let us first introduce some basic notation.

\begin{dfn} \label{dfn:hyppoly} Let ${\bf P}$ be a (closed) hyperbolic
  $k$-gon. Let $x_1,\,\ldots,\,x_k$ (which lie in ${\mathbb D}^2$ or
  its boundary $\partial {\mathbb D}^2$) denote the vertices of
  ${\bf P}$ cyclically ordered anti-clockwise. Let $[x_i,x_{i+1}]$
  denote the (geodesic) side of ${\bf P}$ with endpoints $x_i$ and
  $x_{i+1}$, where indices are taken modulo $k$. We use the convention
  that each side contains both its endpoints.
\end{dfn}

Henceforth, we only consider hyperbolic polygons
${\bf P} \subset {\mathbb D}^2$ with interior angles equal to the
fixed value $\pi/l$ (compact case) or $0$ (ideal case). Such a polygon
gives rise to a tessellation of ${\mathbb D}^2$ via repeated
reflections and to a natural projection map from ${\mathbb D}^2$ to
${\bf P}$. Then the projection of every oriented bi-infinite geodesic
in ${\mathbb D}^2$ can be viewed as a billiard trajectory in
${\bf P}$, as long as the geodesic in ${\mathbb D}^2$ is not
completely contained in the union of the boundaries of the tiles in
this tessellation. Conversely, given a billiard trajectory in
${\bf P}$ with a start point, it can be {\em unfolded} to a
bi-infinite geodesic in ${\mathbb D}^2$ by reflecting the billiard
table along its sides instead of the trajectory, whenever it hits the
boundary. Note that this viewpoint allows us to define billiard
trajectories of ${\bf P}$ even in the case when they hit the vertices
$x_i$ of ${\bf P}$. We also like to mention for the sake of simplicity
that, if there is no danger of misinterpretation, we often do not
distinguish between billiard trajectories and geodesics given as
arc-length parametrized curves and their geometric representation as
subsets of polygons or surfaces.

\begin{dfn}
An \emph{arc} $\alpha$ of a closed billiard trajectory in ${\bf P}$ is 
a closed geodesic arc whose interior lies in ${\bf P}$ and whose endpoints
lie on $\partial{\bf P}$. Note that an endpoint of $\alpha$ may be
a vertex of ${\bf P}$, but such a vertex must lie in ${\mathbb D}^2$.
\end{dfn} 

Now we can introduce the concept of being {\em filling}. The different
types of connected components in Definition \ref{def-fill} are
illustrated in Figure \ref{fig:filling}.

\begin{dfn}\label{def-fill}
  Let ${\bf P}$ be a (closed) hyperbolic polygon and $\gamma$ be a
  union of closed billiard trajectories in ${\bf P}$. We say that
  $\gamma$ \emph{fills} ${\bf P}$ if $\gamma$ is connected and each
  connected component of ${\bf P}\setminus\gamma$ is a topological
  disc whose boundary is one of:
  \begin{itemize}
  \item[(a)] a topological circle in $\gamma$ made up of segments of
    geodesic arcs of $\gamma$;
  \item[(b)] a topological arc in $\gamma$ (made up of segments of
    geodesic arcs in $\gamma$) and an arc of one side of ${\bf P}$,
    possibly including one or both vertices in this side; or
  \item[(c)] a topological arc in $\gamma$, exactly one vertex of
    ${\bf P}$ and an arc in each of the two sides ending in this
    vertex, but not including either of the other vertices in these
    sides.
  \end{itemize}
\end{dfn}

\begin{figure}
  \begin{center}
    \psfrag{P}{${\bf P}$} 
    \includegraphics[width=\textwidth]{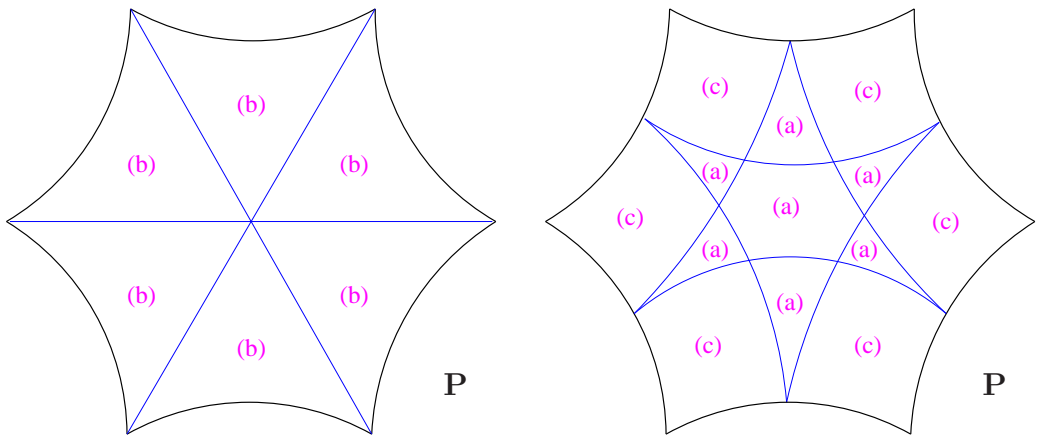}
    \caption{Examples of connected components of type (a), (b) and (c)
    in Definition \ref{def-fill} of ${\bf P} \setminus \gamma$.
    $\gamma$ is the union of the blue arcs. \label{fig:filling}}
  \end{center}
\end{figure}

Now we state the main result of this section.

\begin{prop}\label{prop-fill}
  Let ${\bf P}_0$ be a regular hyperbolic $k$-gon and let $\rho$
  denote the anti-clockwise rotation through angle $2\pi/k$ about the
  centre of ${\bf P}_0$.  Let $\gamma_0$ be a closed billiard
  trajectory in ${\bf P}_0$ and let $\gamma_i=\rho^i(\gamma_0)$ for
  $i=1,\,\ldots,\,k-1$.  Then $\gamma=\bigcup_{i=0}^{k-1}\gamma_i$
  fills ${\bf P}_0$.
\end{prop}

Note that the curves $\gamma_i$ are the closed billiard trajectories
cyclically related to $\gamma_0$. An important lemma for the proof is
the following.

\begin{lm}\label{lem-intersect}
If $\gamma_0$ is a closed billiard trajectory in a hyperbolic
polygon ${\bf P}$ then there are two non-adjacent sides of ${\bf P}$
that intersect $\gamma_0$ (not necessarily as endpoints of a single 
geodesic arc of the trajectory).
\end{lm}

\begin{proof}
We suppose the result is false. That is, suppose that $\gamma_0$
is a closed billiard trajectory in ${\bf P}$ and there are two sides 
$[x_{i-1},x_i]$ and $[x_i,x_{i+1}]$ of ${\bf P}$ so that 
that every arc of $\gamma_0$ has one endpoint in
$[x_{i-1},x_i]$ and the other endpoint in $[x_i,x_{i+1}]$.
Note that $\gamma_0$ cannot pass through $x_i$ since a geodesic arc
from $x_i$ to a point in either of these two sides must be contained 
in this side. Moreover, if $\gamma_0$ passes through $x_{i-1}$ (or $x_{i+1}$) 
then we could find an arc in $\gamma_0$ connecting the non-adjacent sides 
$[x_{i-2},x_{i-1}]$ and $[x_i,x_{i+1}]$ (or the non-adjacent sides
$[x_{i-1},x_i]$ and $[x_{i+1},x_{i+2}]$ respectively).
Let $<$ denote the natural anticlockwise order on 
$[x_{i-1},x_i]\cup[x_i,x_{i+1}]$. 

There are finitely many intersection points of $\gamma_0$ with 
$\partial{\bf P}$. Write them as $y_j$ where $-n\le j\le m$ and $j\neq 0$
where
$$
x_{i-1}<y_{-n}<y_{-n+1}<\cdots<y_{-1}< x_i<
y_1<\cdots<y_{m-1}<y_m<x_{i+1}.
$$
Every geodesic arc in $\gamma_0$ connects a point $y_{-r}$ with negative
index and a point $y_s$ with positive index.

Consider $y_{-n}$. Suppose $a$ is the largest index so that there is
an arc of $\gamma_0$ from $y_{-n}$ to $y_a$ (see Figure
\ref{fig-triangT} for illustration). Then there is a point $y_b$ with
$b\le a$ so that the arc $[y_{-n},y_b]$ is adjacent to $[y_{-n},y_a]$
in the billiard trajectory $\gamma_0$. (Note that there could be other
arcs of $\gamma_0$ with endpoint $y_{-n}$.)  Since the angle of
incidence equals the angle of reflection, the angle $\theta_{-n}$
between the arcs $[y_b,y_{-n}]$ and $[y_{-n},x_i]$ equals the angle
between the arcs $[y_a,y_{-n}]$ and $[y_{-n},x_{i-1}]$.  Since
$b\le a$ then $\theta_{-n}\le \pi/2$.

Similarly for $y_m$. Let $-c$ be the smallest index so that there is
an arc from $y_m$ to $y_{-c}$. Then the angle $\theta_m$ between the
arcs $[y_{-c},y_m]$ and $[y_m,x_{i+1}]$ is at most $\pi/2$.

\begin{figure}
  \begin{center}
    \psfrag{xi-1}{$x_{i-1}$} 
    \psfrag{xi}{$x_i$}
    \psfrag{xi+1}{$x_{i+1}$} 
    \psfrag{p-n}{\blue{$\theta_{-n}$}}
    \psfrag{y-n}{$y_{-n}$} 
    \psfrag{yb}{$y_b$} 
    \psfrag{ya}{$y_a$}
    \psfrag{ym}{$y_m$} 
    \psfrag{T}{\red{$T$}}
    \includegraphics[height=6cm]{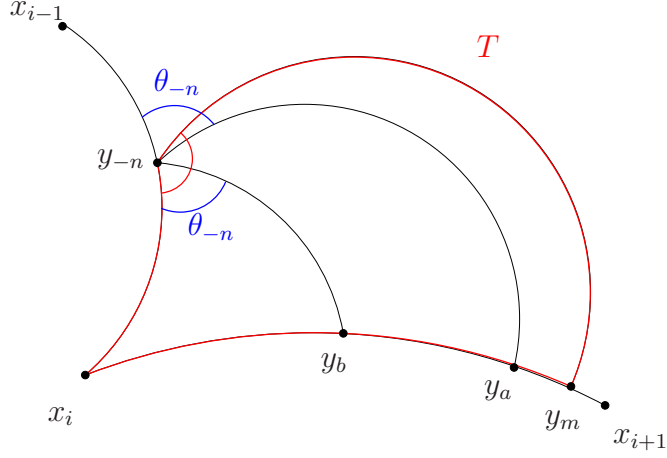}
  \caption{Note that $\theta_{-n} \le \pi/2$ and, therefore, the red
  internal angle of the triangle $T$ at $y_{-n}$ must be at least $\pi/2$.
  \label{fig-triangT}}
  \end{center}
\end{figure}

Now consider the solid closed geodesic triangle $T$ with vertices
$x_i$, $y_m$ and $y_{-n}$.  The entire billiard trajectory must be
contained in $T$. The internal angle of $T$ at $y_{-n}$ is at least
$\pi-\theta_{-n}$ and the internal angle at $y_m$ is at least
$\pi-\theta_m$. But both these angles are at least $\pi/2$, which
contradicts the fact that the sum of internal angles of a hyperbolic
triangle are less than $\pi$.
\end{proof}

\begin{dfn}
Let $\alpha$ and $\beta$ be two closed geodesic arcs in a hyperbolic 
polygon ${\bf P}$ with distinct endpoints. We say that the endpoints 
of $\alpha$ and $\beta$ in $\partial{\bf P}$ \emph{interlace} if each 
interval of $\partial{\bf P}$ between the endpoints of $\alpha$ contains
an endpoint of $\beta$ and vice versa.
\end{dfn}

We leave the easy proof of the following fact to the reader.

\begin{lm}\label{lem-interlace}
Let $\alpha$ and $\beta$ be two closed geodesic arcs in a hyperbolic 
polygon ${\bf P}$ with distinct endpoints. If the endpoints 
of $\alpha$ and $\beta$ interlace then $\alpha$ and $\beta$ intersect
in an interior point of ${\bf P}$. 
\end{lm}

Finally, we give a detailled proof of our main result of this section.

\begin{proof}[Proof of Proposition \ref{prop-fill}]
We begin by proving $\gamma$ is arcwise connected. We divide the proof into 
two cases.

\begin{figure}[h]
  \begin{center}
    \psfrag{xi-2}{$x_{i-2}$}
    \psfrag{xi-1}{$x_{i-1}$} 
    \psfrag{xi}{$x_i$}
    \psfrag{xi+1}{$x_{i+1}$}
    \psfrag{xi+2}{$x_{i+2}$}
    \psfrag{xj-1}{$x_{j-1}$} 
    \psfrag{xj}{$x_j$}
    \psfrag{xj+1}{$x_{j+1}$}
    \psfrag{(a)}{{\Large{(a)}}}
    \psfrag{(b)}{{\Large{(b)}}}
    \psfrag{(c)}{{\Large{(c)}}}
    \psfrag{a0}{\red{$\alpha_0$}}
    \psfrag{a1=r(a0)}{\red{$\alpha_1=\rho(\alpha_0)$}}
    \psfrag{a0-}{\red{$\alpha_0^-$}}
    \psfrag{a0+}{\red{$\alpha_0^+$}}
    \psfrag{b0-}{\red{$\beta_0^-$}}
    \psfrag{b0+}{\red{$\beta_0^+$}}
    \psfrag{a1-}{\blue{$\alpha_1^-$}}
    \psfrag{a1+}{\blue{$\alpha_1^+$}}
    \psfrag{b1-}{\blue{$\beta_1^-$}}
    \psfrag{b1+}{\blue{$\beta_1^+$}}
    \includegraphics[height=10cm]{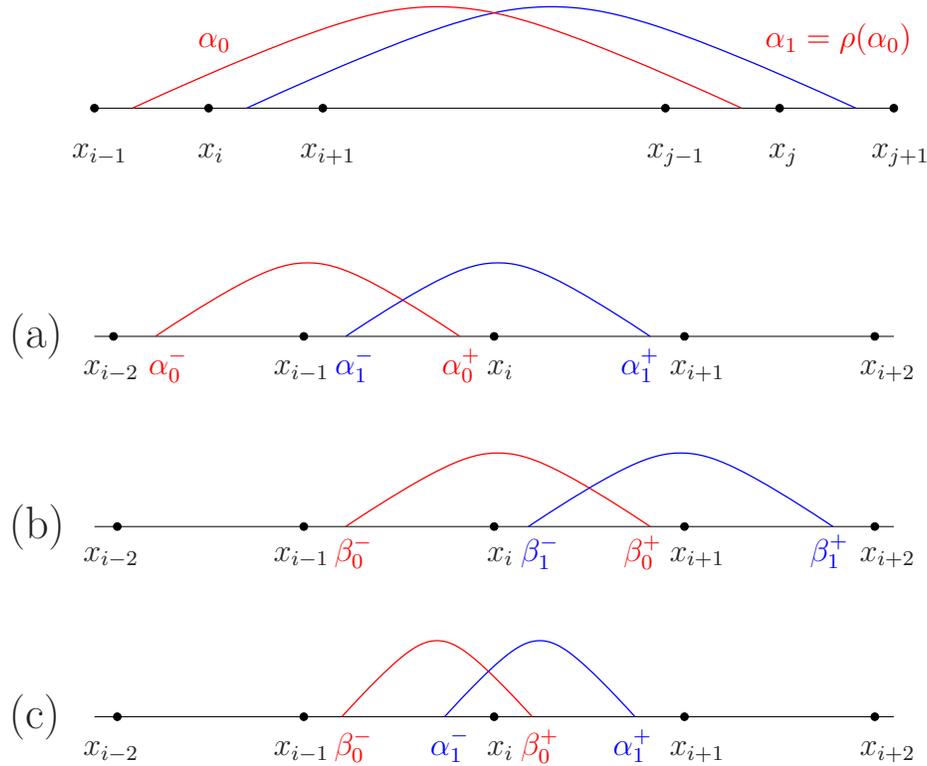}
    \caption{Configurations of interlacing in the proof of Proposition
      \ref{prop-fill}. The boundary $\partial {\bf P}$ is straigthened
    to simplify the illustration.
  \label{fig-interlace}}
  \end{center}
\end{figure}

First, suppose that there is an arc $\alpha_0$ of $\gamma_0$ so that
there are two non-adjacent sides of ${\bf P}_0$ containing the
endpoints of $\alpha_0$. Let $[x_{i-1},x_i]$ and $[x_{j-1},x_j]$
denote these two sides. Since these edges are not adjacent $x_{i-1}$,
$x_i$, $x_{j-1}$ and $x_j$ are all distinct. Now consider
$\alpha_1=\rho(\alpha_0)$. It intersects the boundary of ${\bf P}_0$
in the sides $[x_i,x_{i+1}]$ and $[x_j,x_{j+1}]$.  The intervals
$[x_{i-1},x_i]$, $[x_i,x_{i+1}]$, $[x_{j-1},x_j]$ and $[x_j,x_{j+1}]$
are distinct by construction and occur in this cyclic order.
Therefore, as we move around the boundary of ${\bf P}_0$ the endpoints
of $\alpha_0$ and $\alpha_1$ interlace (see top of Figure
\ref{fig-interlace}). This means that $\alpha_0$ and $\alpha_1$
intersect in ${\bf P}_0$ by Lemma \ref{lem-interlace}.  Hence
$\gamma_0$ and $\gamma_1$ intersect. Applying powers of $\rho$ we see
that $\gamma_i$ and $\gamma_{i+1}$ intersect, thus proving that
$\gamma$ is arcwise connected.

Secondly, suppose that every arc of $\gamma_0$ connects adjacent sides
of ${\bf P}_0$. Every such arc has to connect interior points of the
adjacent sides of ${\bf P}_0$ for, otherwise, we would be in the first
case. Using Lemma \ref{lem-intersect}, we can find consecutive arcs
$\alpha_0$ and $\beta_0$ of $\gamma_0$ meeting $\partial{\bf P}_0$ in
three successive sides. To be precise, suppose one end $\alpha_0^-$ of
$\alpha_0$ is a point in $[x_{i-2},x_{i-1}]$, the common endpoint
$\alpha_0^+=\beta_0^-$ of $\alpha_0$ and $\beta_0$ lies in
$[x_{i-1},x_i]$ and the other endpoint $\beta_0^+$ of $\beta_0$ lies
in $[x_i,x_{i+1}]$. Note that because $\alpha_0$ and $\beta_0$ are
geodesic arcs, their only intersection point is their common
endpoint. They therefore form an m-shaped curve.  Now consider
$\alpha_1\cup\beta_1=\rho(\alpha_0\cup\beta_0)$, with endpoints
$\alpha_1^-\in[x_{i-1},x_i]$, $\alpha_1^+=\beta_1^-\in[x_i,x_{i+1}]$
and $\beta_1^+\in[x_{i+1},x_{i+2}]$. If the sets
$\{\alpha_0^-,\,\alpha_0^+,\,\beta_0^+\}$ and
$\{\alpha_1^-,\,\alpha_1^+,\,\beta_1^+\}$ have a point in common, then
$\gamma_0$ and $\gamma_1$ intersect, and so $\gamma$ is connected as
above.  Thus, we may assume these two sets are disjoint. It suffices
to show that certain endpoints of these arcs interlace and so, using
Lemma \ref{lem-interlace}, the corresponding arcs intersect. If $<$
denotes the natural anti-clockwise order on
$[x_{i-1},x_i]\cup[x_i,x_{i+1}]$, then it is easy to show:
\begin{enumerate}
\item[(a)] if $\alpha_1^-<\alpha_0^+$ then
$\alpha_0^-,\,\alpha_1^-,\,\alpha_0^+,\,\alpha_1^+$ interlace, and so
$\alpha_0$ and $\alpha_1$ intersect;
\item[(b)] if $\beta_1^-<\beta_0^+$ then
$\beta_0^-,\,\beta_1^-,\,\beta_0^+,\,\beta_1^+$ interlace, and so
$\beta_0$ and $\beta_1$ intersect;
\item[(c)] if $\beta_0^-<\alpha_1^-$ and $\beta_0^+<\alpha_1^+$
then $\beta_0^-,\,\alpha_1^-,\,\beta_0^+,\,\alpha_1^+$ interlace,
and so $\beta_0$ and $\alpha_1$ intersect.
\end{enumerate}
The cases (a)-(c) are illustrated in Figure \ref{fig-interlace}. We
observe that, since $\alpha_0^+=\beta_0^-$ and $\alpha_1^+=\beta_1^-$
then condition (c) is precisely the condition that both (a) and (b)
fail.  Therefore these three cases exhaust all possibilities.  The
argument then follows as in the first case.

This shows that $\gamma$ is arcwise connected. Since ${\bf P}_0$ is
topologically a disc every connected component $U$ of
${\bf P}_0\setminus\gamma$ is a topological disc. If every point of
$\partial U$ lies in $\gamma$, then we have case (a) of Definition
\ref{def-fill}. So suppose that $\partial U$ contains a point of
$\partial{\bf P}_0$ that is not contained in $\gamma$.  Then
$\partial U$ contains a non-empty topological arc of
$\partial{\bf P}_0$ both of whose endpoints lie in $\gamma$ (note this
arc is not necessarily contained in just one side of ${\bf P}_0$). We
claim that $\partial U$ can contain at most one such topological arc
in $\partial{\bf P}_0$.  Suppose this is false. Then we can find four
points $b_1$, $b_2$, $c_1$, $c_2$ in $\partial U$ so that (a) the
points $c_1$, $c_2$ lie in $\gamma$, (b) the points $b_1$, $b_2$ lie
in the interior of arcs of $\partial{\bf P}_0$ not intersecting
$\gamma$ and (c) these four points $b_1$, $c_1$, $b_2$, $c_2$ are
interlaced.  Therefore we can find a Jordan arc $\delta$ from $b_1$ to
$b_2$ (that is from $\partial{\bf P}_0$ to itself) in $U$ (except for
its end points) so that the two connected components of
${\bf P}_0\setminus \delta$ each contains a point of $\gamma$, namely
$c_1$ and $c_2$. This contradicts the connectedness of $\gamma$. Hence
$U$ can contain at most one topological arc of $\partial{\bf P}_0$ in
its boundary. Recall, we are assuming such an arc exists, or else we
are in case (a) of Definition \ref{def-fill}.  Call this arc
$\varepsilon$. If the interior of $\varepsilon$ is contained in only
one side of ${\bf P}_0$, then we are in case (b) of Definition
\ref{def-fill}. If the interior of $\varepsilon$ contains points in
precisely two different sides of ${\bf P}_0$, then these two sides
must be adjacent, say $[x_{i-1},x_i]$ and $[x_i,x_{i+1}]$, and their
common vertex $x_i$ must also be contained in the interior of
$\varepsilon$.  In particular, $\gamma$ does not pass through the
vertex $x_i$.  Since $\gamma$ is preserved by the symmetry map $\rho$,
we see that $\gamma$ does not pass through any vertex of ${\bf
  P}_0$. Since $\gamma$ intersects the sides $[x_{i-1},x_i]$ and
$[x_i,x_{i+1}]$ and does not contain their endpoints, it must contain
points of both their interiors.  In particular, $\varepsilon$ does not
contain $x_{i-1}$ or $x_{i+1}$.  Hence we are in case (c) of
Definition \ref{def-fill}.  Finally, suppose that the interior of
$\varepsilon$ contains points from at least three sides of
${\bf P}_0$. As $\varepsilon$ is connected, this means it contains a
whole side of ${\bf P}_0$, which contradicts the fact that, by
symmetry, each side of ${\bf P}_0$ intersects $\gamma$.  Thus, the
only possibilities for $\partial U$ are (a), (b) and (c) from
Definition \ref{def-fill} as required.
\end{proof}

\section{Teichm\"uller space and Fenchel-Nielsen coordinates}

From now on we fix $k \ge 3$ and we only consider \emph{ideal}
$k$-gons. In contrast to the convention in the previous section, our
ideal $k$-gons ${\bf P}$ do not contain the vertices at infinity, but
they contain the sides and are therefore closed subsets of
${\mathbb D}^2$. A key observation in the proof of Theorem
\ref{thm:main} is that every ideal $k$-gon ${\bf P}$ gives rise to a
Riemann surface $S_{\bf P}$ (its billiard surface) via a gluing
process of two copies of ${\bf P}$, denoted by ${\bf P}^+$ and
${\bf P}^-$, along corresponding sides, and that every closed billiard
trajectory in ${\bf P}$ gives rise to one or two closed geodesics in
$S_{\bf P}$. This allows us to apply powerful results from
Teichm\"uller theory.

Let us first set up the Teichm\"uller space framework and introduce
the relevant objects. A {\rm Riemann surface (of finite type)} $S$ is
a $2$-dimensional oriented differentiable manifold with finitely many
ends, carrying a Riemannian metric of constant curvature minus one. We
suppose that $S$ has finite area with respect to this metric. In
particular, the ends are realised as cusps. The universal covering of
$S$ agrees with ${\mathbb D}^2$ and the canonical complex structure of
${\mathbb D}^2$ induces a complex structure on $S$. Thus it makes
sense to consider holomorphic and anti-holomorphic isometries of $S$.

Let ${\bf P}\subset {\mathbb D}^2$ be an ideal $k$-gon (with
anti-clockwise enumerated vertices $x_1,\dots,x_k$ as in Definition
\ref {dfn:hyppoly}) and $S_{\bf P}$ be the corresponding billiard
surface. Note that $S_{\bf P}$ is, topologically, homeomorphic to a
$k$-punctured sphere. The ends of $S_{\bf P}$ correspond to the
vertices $x_j$. Note that $S_{\bf P}$ is a {\em labelled billiard
  surface} since its ends carry labels in
$\{1,2,\dots,k\}$. Similarly, ${\bf P}$ is a {\em labelled polygon}
where the bi-infinite geodesic side $(x_i,x_{i+1})$ of ${\bf P}$
without its end points is endowed with the label $i$ (mod $k$).
$S_{\bf P}$ has a natural anti-holomorphic isometry
$J_{\bf P}: S_{\bf P} \to S_{\bf P}$ interchanging ${\bf P}^+$ and
${\bf P}^-$, and with fixed point set $\cup_{i=1}^k (x_i,x_{i+1})$.
Let ${\bf P}_0 \subset {\mathbb D}^2$ be an ideal regular $k$-gon with
anti-clockwise labelling, $R_0 = S_{{\bf P}_0}$ be the corresponding
labelled billiard surface, and $J_0 = J_{{\bf P}_0}: R_0 \to R_0$ be the
corresponding anti-holomorphic isometry.

\begin{dfn} The \emph{Teichm\"uller space} $T(R_0)$ is the set of all
  equivalence classes of pairs $(S,f)$ where $S$ is an oriented
  Riemann surface and $f: R_0 \to S$ is a quasiconformal mapping. Two
  such pairs $(S,f)$ and $(S',f')$ are equivalent if the map
  $f' \circ f^{-1}: S \to S'$ is homotopic to an orientation
  preserving isometry. We denote the equivalence class associated to
  the pair $(S,f)$ by $[S,f]$. A point $[S,f]$ in Teichm\"uller space
  $T(R_0)$ is also called a {\em marked Riemann surface}.
\end{dfn}

The Teichm\"uller space $T(R_0)$ carries a natural complex manifold
structure and the anti-holomorphic isometry $J_0: R_0 \to R_0$ gives
rise to an anti-holomorphic automorphism ${\mathcal F}$ on
Teichm\"uller space (see \cite[pp. 229]{IT}), which we call the flip
map.

\begin{dfn} \label{def:indandflip} Let $\varphi: R_0 \to R_0$ be an
  orientation preserving quasiconformal mapping. Then we define the
  induced map $\varphi_*: T(R_0) \to T(R_0)$ as
  $$ \varphi_*([S,f:R_0 \to S]) = [S, f \circ \varphi: R_0 \to S]. $$
  The {\em flip map} ${\mathcal F}: T(R_0) \to T(R_0)$ is defined as
  \begin{equation} \label{eq:flip} 
  {\mathcal F}([S,f: R_0 \to S]) = [S^*, j_S \circ f \circ J_0:
    R_0 \to S^*]. 
  \end{equation}
  Here, $S^*$ is the same surface as $S$ but with the opposite
  orientation and $j_S: S \to S^*$ is, as a map, the pointwise identity.
\end{dfn}

Let $\rho: {\bf P}_0 \to {\bf P}_0$ be the anti-clockwise rotation
through angle $2\pi/k$ about the centre of ${\bf P}_0$. By abuse of
notation, we denote the associated rotation in the corresponding
billiard surface, again, by $\rho: R_0 \to R_0$. The induced map
$\rho_*: T(R_0) \to T(R_0)$ has order $k$. Note that the special point
$x_0 = [R_0, {\rm id}: R_0 \to R_0] \in T(R_0)$ is a common fixed
point of both $\rho_*$ and the flip map ${\mathcal F}$.

Our next aim is to introduce suitable Fenchel-Nielsen coordinates
$(l,\tau)$, which yield a diffeomorphism between $T(R_0)$ and
$({\mathbb R}^+)^{k-3} \times {\mathbb R}^{k-3}$. We first decompose
${\bf P}_0$ into right angled compact hexagons, right angled pentagons
with one ideal vertex and right angled quadrilaterals with two ideal
vertices. Such a decomposition induces a decomposition of $R_0$ into
$k-2$ pairs of pants $Y_1,\dots,Y_{k-2}$ with three/two/one geodesic
boundary cycles, respectively. Each of these pairs of pants $Y_j$ is
invariant (as a set) under the reflection $J_0$, and they have their
own reflections $J_{Y_j}$ which agree with the restrictions of $J_0$
to $Y_j$. For illustration, we now use the following colour
convention: The $k-3$ boundary cycles $C_1,\dots,C_{k-3} \subset R_0$
of the pants decomposition $R_0 = \bigcup_{j=1}^{k-2} Y_j$
are {\em green lines}. The bi-infinite geodesics
$(x_i,x_{i+1}) \subset R_0$ are {\em red lines}. Cutting $R_0$ along
all red lines splits the surface into the two polygons ${{\bf P}_0}^+$
and ${{\bf P}_0}^-$.

\begin{figure}[h]
  \begin{center}
    \psfrag{c4}{$\red{f(c_4)}$} 
    \psfrag{c2}{$\red{f(c_2)}$}
    \psfrag{Y1}{$Y_1'$} 
    \psfrag{Y2}{$Y_2'$}
    \psfrag{t1}{$\green{\tau_1}$} 
    \psfrag{A}{$\red{A}$} 
    \psfrag{B}{$\red{B}$}
    \psfrag{C1}{$\green{C_1'}$}
    \includegraphics[height=7cm]{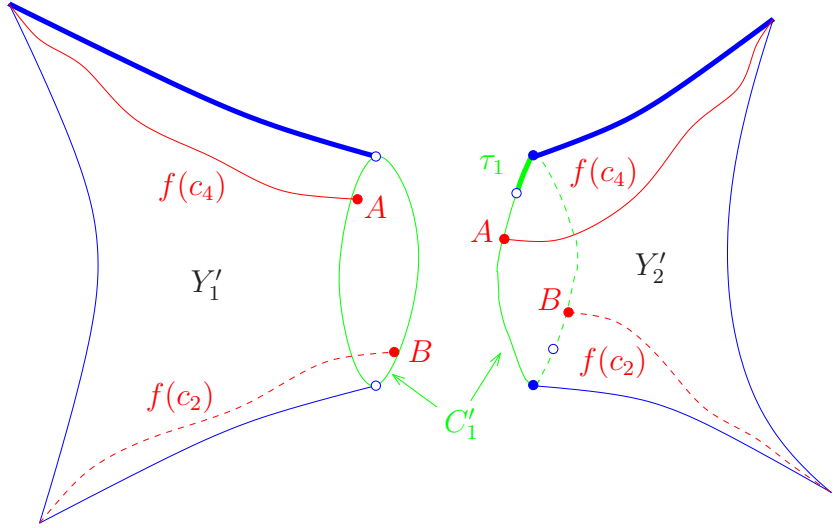}
    \caption{The green boundary cycles of $Y_1'$ and $Y_2'$ are
      identified such that the $A$'s and $B$'s fit together. The red
      curve $f(c_4)$ is freely homotopic (fixing the end points) to
      the union of the two thick blue arcs and the thick green arc of
      length $\tau_1$ in between.\label{fig-fenchniel}}
  \end{center}
\end{figure}

The Fenchel-Nielsen coordinates of a point $[S,f] \in T(R_0)$ are now
given as follows: Let $C_j' \subset S$ be the unique closed geodesic
corresponding to $f(C_j)$ modulo free homotopy in $S$.  Again, we
think of the curves $C_j'$ as {\em green lines}. They give rise to a
pants decomposition $S = \bigcup_{j=1}^{k-2} Y_j'$ agreeing,
combinatorially, with the pants decomposition of $R_0$. The
length parameters of $[S,f]$ are then given by the lengths
$l_j \in {\mathbb R}^+$ of the boundary cycles $C_j' \subset S$.
 
For every geodesic $c_i = (x_i,x_{i+1}) \subset R_0$ of $R_0$, let its
image $f(c_i) \subset S$ again carry the {\em colour red}. Note that
the bi-infinite curves $f(c_i) \subset S$ are generally no longer
geodesics. Each pair of pants $Y_j'$ in the decomposition of $S$ comes
equipped with a triplet of {\em blue geodesic arcs}, namely the fixed
point set of the intrinsic reflection $J_{Y_j'}$ of this pair of
pants. Now, for every bi-infinite red curve $f(c_i) \subset S$ there
exists a unique regular freely homotopic curve connecting the same
ends, which is made up of alternating blue and green arcs (regular
means here that we do not allow to go back and forth in certain parts
of the curve). This means that the curve $f(c_i)$ defines an arc in
each green boundary cycle $C_j'$ along its path, and the length of
this arc provides a unique twist parameter $\tau_j \in {\mathbb
  R}$. Note that the sign of the twist parameter is uniquely
determined by the orientations of the pairs of pants and their
boundary cycles. Note also, that every boundary cycle $C_j'$ defines
an $X$-piece (two pairs of pants glued along $C_j'$) and there are two
curves $f(c_{i_1})$ and $f(c_{i_2})$ intersecting it. The twist
parameter $\tau_j \in {\mathbb R}$ is independent of the choice of
$f(c_{i_1})$ or $f(c_{i_2})$ (see Figure \ref{fig-fenchniel} for
illustration of the twist parameter $\tau_1$ in $S = Y_1' \cup
Y_2'$). For further details we refer to, e.g., \cite[Section
7.6]{Hub}.

\begin{dfn}
  We denote by $B(R_0)$ the subset of $T(R_0)$ with vanishing twist
  parameters. The points $[S,f] \in B(R_0)$ are called {\em marked
    billiard surfaces}. We refer to $B(R_0)$ as the {\em billiard
    space} associated to ${\bf P}_0$.
\end{dfn}

\section{Properties of the billiard space}
\label{sec:propbillspace}

Now we explain that each point of $B(R_0)$ can be realised by a
labelled billiard surface $S$ together with an (almost canonical)
quasiconformal mapping $f: R_0 \to S$, respecting the labelling (i.e.,
mapping the $i$-th end of $R_0$ to the $i$-th end of $S$, for
$i=1,\dots,k$). Given the length coordinates $(l_1,\dots,l_{k-3})$, we
can construct an ideal hyperbolic $k$-gon ${\bf P}$ with these
parameters in its decomposition into hexagons, pentagons and
quadrilaterals consistent with the decomposition of ${\bf P}_0$.
Next, we choose quasiconformal maps from each building block
(hexagon/pentagon/quadrilateral) of ${\bf P}_0$ to the corresponding
building block of ${\bf P}$ mapping corresponding boundary components
onto each other such that they can be combined to a global
quasiconformal map $f_{\bf P}: R_0 \to S_{\bf P}$, equivariant under
the global reflections $J_0$ and $J_{S_{\bf P}}$:
\begin{equation} \label{eq:equiv} 
f_{\bf P} \circ J_0 = J_{S_{\bf P}} \circ f_{\bf P}. 
\end{equation}
By construction, the union $\bigcup_{i=1}^k c_i$ of the bi-infinite
{\em red lines} of $R_0$ are mapped under $f_{\bf P}$ onto the union
of the {\em blue geodesic arcs} of the pairs of pants $Y_j'$ of
$S_{\bf P}$, and the {\em green} boundary cycles of the pants
decomposition of $R_0$ are mapped under $f_{\bf P}$ onto the
corresponding {\em green} boundary cycles of the pants
decomposition of $S_{\bf P}$. This fact guarantees that all twist
parameters of $[S_{\bf P},f_{\bf P}]$ are zero. In this context, we
can think of $B(R_0)$ as the ``subset of labellel billiard surfaces''
in $T(R_0)$.

\begin{remark} \label{rem:identifyx0}
  As seen above, a general point $x \in B(R_0)$ is an equivalence
  class $x = [S_{\bf P}, f_{\bf P}: R_0 \to S_{\bf P}]$ with an almost
  canonical quasiconformal mapping $f_{\bf P}$. Note that
  $x = [S_{\bf P}, f_{\bf P}]$ agrees with
  $x_0 = [R_0, {\rm id}: R_0 \to R_0] \in B(R_0)$ if and only if the
  polygon ${\bf P}$ is regular.
\end{remark}

The Teichm\"uller space $T(R_0)$ carries a complex manifold structure with
a natural symplectic form $\omega_{WP}$, the {\em Weil-Petersson
  symplectic form}. By Wolpert's Theorem (see \cite{Wol1}),
$\omega_{WP}$ can be written in terms of the Fenchel-Nielsen
coordinates $(l,\tau)$ of $T(R_0)$ as follows:
$$ \omega_{WP} = - \sum_{j=1}^{k-3} d\tau_j \wedge dl_j. $$
The symplectic form $\omega_{WP}$ and the almost complex structure on
$T(R_0)$ induce a K\"ahler metric $g_{WP}$ on $T(R_0)$, the {\em
  Weil-Petersson metric}. While the Riemannian metric $g_{WP}$ is
generally not complete (see \cite{Wol75}), it is still true that any
pair of points $x_1, x_2 \in T(R_0)$ can be joined by a unique
Weil-Petersson geodesic (see \cite{Wol87}). The billiard space
$B(R_0)$ has the following useful properties.

\begin{prop} The billiard space $B(R_0)$ is a Lagrangian submanifold
  of the symplectic manifold $(T(R_0),\omega_{WP})$. Moreover,
  $B(R_0)$ is a geodesically convex subset of $(T(R_0),g_{WP})$, i.e.,
  for given $x_1, x_2 \in B(R_0)$ the unique Weil-Petersson geodesic
  connecting $x_1$ and $x_2$ lies entirely in $B(R_0)$.
\end{prop}

\begin{proof} The flip map ${\mathcal F}$, defined in \eqref{eq:flip},
  is an isometry with respect to $g_{WP}$ (see
  \cite[p. 230]{IT}). Written in our Fenchel-Nielsen coordinates we
  have 
  \begin{equation} \label{eq:flipFenchNiel} 
  {\mathcal F}(l,\tau) = (l,-\tau).
  \end{equation}
  This follows from \cite[p. 230, bottom formula]{IT} and the fact
  that $x_0 = [R_0, {\rm id}]$ is a fixed point of ${\mathcal
    F}$. Therefore, the fixed point set of ${\mathcal F}$ is the space
  $B(R_0)$ of marked billiard surfaces. By the above considerations,
  $B(R_0)$ is a Lagrangian submanifold and, as the fixed point set of
  an isometry, $B(R_0)$ is geodesically convex.
\end{proof}

Finally, we give an important characterisation of the point $x_0 \in B(R_0)$.

\begin{prop} \label{prop:charx0}
The only simultaneous fixed point of $\rho_*$ and ${\mathcal F}$
in $T(R_0)$ is $x_0 = [R_0,{\rm id}] \in B(R_0)$.
\end{prop}

\begin{proof} Let $x \in T(R_0)$ be a simultaneous fixed point of 
${\mathcal F}$ and $\rho_*$. 

The fixed point property ${\mathcal F}(x)=x$ and
\eqref{eq:flipFenchNiel} imply that $x \in B(R_0)$. Therefore, $x$ has
a representation $x = [S_{\bf P},f_{\bf P}: R_0 \to S_{\bf P}]$ with
${\bf P} \subset {\mathbb D}^2$ a labelled ideal hyperbolic $k$-gon
and $f_{\bf P} \circ J_0 = J_{S_{\bf P}} \circ f_{\bf P}$. By Remark
\ref{rem:identifyx0}, we only have to show that ${\bf P}$ is regular.

The fixed point property $\rho_*(x) = x$ means that
$g_0 := f_{\bf P} \circ \rho \circ f_{\bf P}^{-1}: S_{\bf P} \to
S_{\bf P}$ is homotopic to an isometry $g_1: S_{\bf P} \to S_{\bf
  P}$. Since a homotopy between two maps on $S_{\bf P}$ preserves the
ends, $g_1$ maps the end $j$ of $S_{\bf P}$ to the end $j+1$
(modulo $k$). Let $c_j$ be the unique geodesic in $S_{\bf P}$
connecting the ends $j$ and $j+1$ modulo $k$. Then the set
$C = \bigcup_{j=1}^k c_j$ splits $S_{\bf P}$ into the ideal polygons
${\bf P}^+$ and ${\bf P}^-$, both isometric to ${\bf P}$, and we have
$g_1(c_j) = c_{j+1}$ for all $j$, modulo $k$.  This means that
$g_1(C) = C$ and the isometry $g_1$ either interchanges ${\bf P}^+$
and ${\bf P}^-$ or preserves them as sets. Since $g_1$ is orientation
preserving, we cannot have $g_1({\bf P}^+) = {\bf P}^-$. This
shows that we have $g_1: {\bf P}^+ \to {\bf P}^+$.

Now we embed ${\bf P}^+$ into ${\mathbb D}^2$ and compactify
${\bf P}^+$ by adding the ideal vertices
$x_1,\dots,x_k \in \partial {\mathbb D}^2$ corresponding to the ends
$1,\dots,k$, respectively. Then the isometry $g_1$ extends to a
continuous map, denoted again by $g_1$, of the compacification
$\overline{{\bf P}^+}$. By Brouwer's Fixed Point Theorem, there exists
$z_0 \in \overline{{\bf P}^+}$ such that $g_1(z_0)=z_0$. This point
must be an interior point of $\overline{{\bf P}^+}$ since the boundary
of $\overline{{\bf P}^+}$, consisting of the points $x_1,\dots,x_k$
and the geodesics $c_j$, cannot have a fixed point (recall that $g_1$
maps $x_j$ to $x_{j+1}$ modulo $k$).  Let $r_j$ be the geodesic ray
connecting $z_0$ with the ideal point $x_j$. Then $g_1$ maps the
triangle with vertices $z_0,x_j,x_{j+1}$ to the triangle with vertices
$z_0,x_{j+1},x_{j+2}$ modulo $k$. Therefore, all the triangles with
vertices $z_0,x_j,x_{j+1}$ for $j=1,\dots,k$ are isometric to one
another. Since isometries preserve angles, the angle between the rays
$r_j$ and $r_{j+1}$ at $z_0$ must therefore be $2\pi/k$. This shows
that ${\bf P}^+ \subset {\mathbb D}^2$ is a regular $k$-gon, finishing
the proof.
\end{proof}

\section{Geodesic lengths functions and cyclically related billiard
  trajectories}

Recall that ${\bf P}_0 \subset {\mathbb D}^2$ denotes a labelled
regular ideal $k$-gon and that $R_0$ is its associated billiard
surface with rotational symmetry $\rho: R_0 \to R_0$. Let us now
introduce geodesic length functions on Teichm\"uller space.

\begin{dfn} \label{dfn:geodlengthf} A closed curve in $R_0$ is called
  {\em essential} if it is neither null-homotopic nor spirals around
  one of the ends of $R_0$. Let
  $\widetilde \gamma = \{ \widetilde \gamma_1, \dots, \widetilde
  \gamma_N \}$ be a finite family of essential closed curves
  $\widetilde \gamma_i: S^1 \to R_0$. For $x=[S,f] \in T(R_0)$, let
  $\widetilde \gamma_i'$ be the unique closed geodesic which is freely
  homotopic to $f(\widetilde \gamma_i)$. Then the {\em geodesic length
    function} associated to $\widetilde \gamma$ is a map
  $$ L = L_{\widetilde \gamma}: T(R_0) \to [0,\infty), $$
  defined by
  $$ L(x) = \sum_{i=1}^N {\rm length}(\widetilde \gamma_i'). $$
\end{dfn}

Note that we can continuously deform a curve spiralling around one of
the ends of a hyperbolic surface into an arbitrarily short curve by
moving it up into the end. This is the reason that such curves are not
considered to be essential. The following fundamental convexity result
of Wolpert will be key for the proof of Theorem \ref{thm:main}.

\begin{thm}[{\cite[Cor. 4.7]{Wol87}}] \label{thm:wolpert} Let
  $\widetilde \gamma \subset R_0$ be a finite family of essential
  closed curves and $L = L_{\widetilde \gamma}: T(R_0) \to (0,\infty)$
  be the associated geodesic length function. Then the function $L$ is
  continuous and \emph{strictly convex} along every Weil-Petersson
  geodesic.
\end{thm}

Let us now link this concept with cyclically related closed billiard
trajectories in different ideal hyperbolic $k$-gons. This requires
further notation.

Let ${\bf P} \subset {\mathbb D}^2$ be a labelled ideal $k$-gon. It
was shown in \cite[Thm 2.1]{CPS} that a finite sequence
${\bf a} = (a_0,a_1,\dots,a_{n-1})$ is a billiard sequence (i.e., a
coding of a closed billiard trajectory) if and only if (a) consecutive values
$a_j$ and $a_{j+1}$ with indices taken modulo $n$ do not coincide and
(b) if ${\bf a}$ contains only two different labels, then they must
not be neighbours (i.e., must not differ by $\pm 1$ modulo $k$). Let
$\gamma_{{\bf a},{\bf P}}$ be the family consisting of the unique
closed billiard trajectory associated to $\bf a$ and all its cyclically
related billiard trajectories in $\bf P$. Then
$\gamma_{{\bf a},{\bf P}}$ consists of $k$ piecewise geodesic closed
curves $\gamma_i$. Let $L_{av}({\bf P}, {\bf a})$ be the average
length of these curves, i.e.,
$$ L_{av}({\bf P}, {\bf a}) = 
\frac{1}{k} \sum_{\gamma_i \in \gamma_{{\bf a},{\bf P}}} {\rm length}(\gamma_i).
$$
Let $\pi_{\bf P}: S_{\bf P} \to {\bf P}$ be the canonical projection
and
$\widetilde \gamma_{{\bf a},{\bf P}} = \pi_{\bf P}^{-1}(\gamma_{{\bf
    a},{\bf P}})$ be the lift of these billiard trajectories in the
corresponding billiard surface. Note that
$\widetilde \gamma_{{\bf a},{\bf P}}$ consists of $2k$ or $k$ closed
geodesics in $S_{\bf P}$, depending on whether $n$ is even or odd: Let
$\widetilde \gamma_i$ be one of the closed geodesics in
$\widetilde \gamma_{{\bf a},{\bf P}}$.  Then there exists a fixed
integer $t$, such that every label $s=a_j$ corresponds to a
transversal crossing between $\widetilde \gamma_i$ and a bi-infinite
geodesic $(x_{s+t},x_{s+t+1})$ (where indices are taken modulo $k$),
i.e., $\widetilde \gamma_i$ changes from ${\bf P}^{\pm}$ to
${\bf P}^{\mp}$. After $n$ such changes $\widetilde \gamma_i$ will not
close up if $n$ is odd. This is the reason why, in this case,
$\widetilde \gamma_{{\bf a},{\bf P}}$ consists of $k$ geodesics
corresponding to cutting sequences\footnote{As in the case of a
  labelled polygon $\bf P$, we can associate a symbolic coding to a
  closed curve in a labelled billiard surface $S_{\bf P}$ reflecting
  its crossings with the bi-infinite geodesics connecting subsequent
  ends. We refer to it as the \emph{cutting sequence} associated to
  the curve.} cyclically related to the doubling
${\bf a}{\bf a}=(a_0,a_1,\dots,a_{n-1},a_0,\dots,a_{n-1})$. But it is
obvious that we have in both cases
$$ \sum_{\widetilde \gamma_i \in \widetilde \gamma_{{\bf a},{\bf P}}} 
{\rm length}(\widetilde \gamma_i) = 2 k L_{av}({\bf P}, {\bf a}). $$
Moreover, the left hand side can be rewritten as the geodesic length
function associated to
$\widetilde \gamma_{\bf a} = \widetilde \gamma_{{\bf a},{\bf P}_0}$,
i.e.,
\begin{equation} \label{eq:avlength-teichm}
L_{\widetilde \gamma_{\bf a}}([S_{\bf P},f_{\bf P}]) = 2 k L_{av}({\bf P}, {\bf a}),
\end{equation}
with $f_{\bf P}: R_0 \to S_{\bf P}$ introduced at the beginning of
Section \ref{sec:propbillspace}. Note here that each closed curve in
$f_{\bf P}(\widetilde \gamma_{\bf a})$ is freely homotopic to a
corresponding curve in the family $\widetilde \gamma_{{\bf a},{\bf P}}$
since both closed curves in $S_{\bf P}$ have the same cutting
sequences.

We finish this section with the following useful observation.

\begin{lm} \label{lm:inv-length}
  Let ${\bf a} = (a_0,\dots,a_{n-1})$ be a billiard sequence. Then we
  have, for all $x \in T(R_0),$
  $$ L_{\widetilde \gamma_{\bf a}}(x) = L_{\widetilde \gamma_{\bf a}}(\rho_*(x)) = 
  L_{\widetilde \gamma_{\bf a}}({\mathcal F}(x)). $$
\end{lm}

\begin{proof}
  Note that
  $\widetilde \gamma_{\bf a} = \{ \widetilde \gamma_1,\dots,
  \widetilde \gamma_N \}$ with $N=k$ or $N=2k$ is a family of closed
  geodesics in $R_0$ which, as a set, is invariant under $\rho$ and
  $J_0$ by its very construction. If $x = [S,f:R_0 \to S]$ then
  $\rho_*(x) = [S,f \circ \rho: R_0 \to S]$ and we have
  $$ L_{\widetilde \gamma_{\bf a}}(\rho_*(x)) = 
  \sum_{i=1}^N {\rm length}(\widetilde \gamma_i'), $$ 
  where $\widetilde \gamma_i'$ is freely homotopic to
  $f \circ \rho(\widetilde \gamma_i)$. The result for $\rho$ follows
  now from the the fact that $\rho$ only permutes the closed curves
  $\widetilde \gamma_i$. The result for the flip map $\mathcal F$ follows
  analogously from the fact that also $J_P$ only permutes the closed
  curves $\widetilde \gamma_i$.
\end{proof}

\section{Proof of Theorem \ref{thm:main}}

As before, let ${\bf P}_0 \subset {\mathbb D}^2$ be a labelled regular
ideal $k$-gon and ${\bf a}$ be a finite billiard sequence. Theorem
\ref{thm:main} in the Introduction states that
\begin{equation} \label{eq:mainres}
L_{av}({\bf P},{\bf a}) \ge L_{av}({\bf P}_0,{\bf a}) 
\end{equation}
for all ideal $k$-gons ${\bf P} \subset {\mathbb D}^2$ with equality
if and only if $P$ is regular. Recall that $\widetilde \gamma_{\bf a}$ is a
family of closed geodesics in $R_0$ associated to the billiard
sequence $\bf a$ and that
$x_0 = [R_0, {\rm id}] \in B(R_0) \subset T(R_0)$. Then
\eqref{eq:mainres} is a consequence of the following, by identity
\eqref{eq:avlength-teichm}: For any finite billiard sequence
${\bf a}$, the geodesic length function associated to
$\widetilde \gamma_{\bf a}$ satisfies
\begin{equation} \label{eq:mainres-teichm} 
L_{\widetilde \gamma_{\bf a}}(x) \ge L_{\widetilde \gamma_{\bf a}}(x_0), 
\end{equation}
with equality iff $x=x_0$. So our goal is to prove \eqref{eq:mainres-teichm}.

Let us return to the property of closed curves to be filling, but now
in the setting of the Riemann surface $R_0$.

\begin{dfn} A family of closed curves
  $\{ \widetilde \gamma_1,\dots,\widetilde \gamma_N \}$ in $R_0$ is
  called {\em filling} if each connected component of
  $R_0 \backslash \bigcup_{i=1}^N \widetilde \gamma_i$ is
  topologically an open disc or a once-puctured open disc.
\end{dfn}

The importance of being filling becomes clear in the following result
by Kerckhoff.

\begin{prop}[{\cite[Lemma 3.1]{Ker83}}] \label{prop:proper} Let
  $\{ \widetilde \gamma_1,\dots,\widetilde \gamma_N \}$ be a finite
  family of closed curves and $L: T(R_0) \to (0,\infty)$ be the
  associated geodesic length function, introduced in Definition
  \ref{dfn:geodlengthf}. If this family is filling, then $L$ is a
  \emph{proper} function.
\end{prop}

Proposition \ref{prop:proper} and Theorem \ref{thm:wolpert} together
imply the following corollary. The proof of this corollary is well
known (see, e.g., last paragraph of \cite{Wol87} or also \cite[Thm
3]{Ker83}) but we include it here for the reader's convenience.

\begin{cor} \label{cor:uniqmin}
  Let$\{ \widetilde \gamma_1,\dots,\widetilde \gamma_N \}$ be a finite
  family of closed essential curves which is filling and
  $L: T(R_0) \to (0,\infty)$ be the associated geodesic length
  function. Then there is a unique point $x_{min} \in T(R_0)$ where
  $L$ assumes its global minimum.
\end{cor} 

\begin{proof}
  Let $L_0 = \inf\{ L(x) \mid x \in T(R_0) \} \ge 0$ and
  $x_m \in T(R_0)$ be a sequence satisfying
  $\lim_{m \to \infty} L(x_m) \to L_0$. Since $L$ is proper, by
  Proposition \ref{prop:proper}, $L^{-1}([0,L_0+1])$ is compact and
  there exists a convergent subsequence
  $x_{m_j} \to x_{min} \in T(R_0)$ with
$$ 0 < L(x_{min}) = \lim_{j \to \infty} L(x_{m_j}) = L_0. $$ 
Assume we have another point $x' \in T(R_0)$ with $L(x') = L_0$. Then
there exists a unique geodesic connecting $x_{min}$ and $x'$, along
which $L$ is strict convex, by Theorem \ref{thm:wolpert}. This would
lead to a point $x'' \in T(R_0)$ between $x_{min}$ and $x'$ with
$L(x'') < L(x_{min}) = L_0$, which is a contradiction.
\end{proof}

Let us, finally, present the proof of \eqref{eq:mainres-teichm}: We
first explain why
$\widetilde \gamma_{\bf a} = \{ \widetilde \gamma_1, \dots, \widetilde
\gamma_N \}$ with $N = k$ or $N=2k$ is filling in $R_0$. We know from
Proposition \ref{prop-fill} that if $\gamma_0$ denotes the closed
billiard trajectory corresponding to $\bf a$ in $\bf P_0$ and
$\gamma_i = \rho^i(\gamma_0)$, then
$\gamma = \bigcup_{i=0}^{k-1} \gamma_i$ fills ${\bf P}_0$. Now, $R_0$
consists of two copies ${\bf P}_0^\pm$ of ${\bf P_0}$, glued along
their boundaries. Under the identification ${\bf P}_0 = {\bf P}_0^+$,
we have
$$ R_0 \backslash \bigcup_{i=1}^N \widetilde \gamma_i = 
\left( {\bf P}_0 \backslash \bigcup_{i=0}^{k-1} \gamma_i \right) \cup
J_0 \left( {\bf P}_0 \backslash \bigcup_{i=0}^{k-1} \gamma_i
\right), $$ and from the domains with properties (a), (b), (c) in
Definition \ref{def-fill} it is easy to see that the connected
components of $R_0 \backslash \bigcup_{i=1}^N \widetilde \gamma_i$ are
either topologically an open disc or a once-puctured open disc. This
shows that $\widetilde \gamma_{\bf a}$ is filling. Moreover, the
geodesics $\widetilde \gamma_i$ are essential and we conclude from
Corollary \ref{cor:uniqmin} that there exists a \emph{unique} point
$x_{min} \in T(R_0)$ with
$$ L(x) > L(x_{min}) \qquad \text{for all $x \in T(R_0), x \neq x_{min}$,} $$ 
where $L$ denotes the geodesic length function associated to
$\widetilde \gamma_{\bf a}$. It only remains to identify this global
minimum.  We know from Lemma \ref{lm:inv-length} that
$L(x) = L(\rho_*(x)) = L({\mathcal F}(x))$, and the uniqueness of the
minimum implies that we have
$$ x_{min} = \rho_*(x_{min}) = {\mathcal F}(x_{min}). $$
It then follows from Proposition \ref{prop:charx0} that we must have
$x_0=x_{min}$.  \qed

\appendix

\section{Billiard in Euclidean rectangles}
\label{app:euclrect}

In this appendix, we discuss an Euclidean analogue of the conjecture,
namely we consider lengths of cyclically related closed billiard trajectories in
Euclidean rectangles of area one. For every $c > 0$, we introduce the
rectangular billiard table
${\bf P}_{c} = [0,c] \times [0,1/c] \subset {\mathbb R}^2$. Every
closed billiard trajectory in ${\bf P}_c$ is, up to free homotopy, in
one-one correspondence with a vector $(nc,m/c)$ with
$(n,m) \in {\mathbb Z}^2\backslash \{(0,0)\}$. The closed billiard
trajectories cyclically related to $(nc,m/c)$ are $(-mc,n/c)$,
$(-nc,-m/c)$ and $(mc,-n/c)$. The lengths of these four cyclically
related billiard trajectories add up to
$$
L_{n,m}(c) = 2 \sqrt{n^2 c^2 + \frac{m^2}{c^2} } + 2 \sqrt{m^2 c^2 +
  \frac{n^2}{c^2} }.
$$  
The Euclidean analogue of the conjecture in this ''baby case'' then
reads as
\begin{equation} \label{eq:babybill} L_{n,m}(c) \ge L_{n,m}(1),
\end{equation}
with equality if and only if $c=1$. \eqref{eq:babybill} is equivalent
to
$$ \sqrt{n^2 c^2 + \frac{m^2}{c^2} } + \sqrt{m^2 c^2 + \frac{n^2}{c^2} } \ge 2 \sqrt{n^2 + m^2}. $$
Squaring both sides leads to
$$ 2 \sqrt{n^2 c^2 + \frac{m^2}{c^2} } \sqrt{m^2 c^2 + \frac{n^2}{c^2} } \ge 2 (n^2+m^2) - \left( c- \frac{1}{c} \right) (n^2 + m^2). $$
This shows that we have \eqref{eq:babybill} if
$$ \sqrt{n^2 c^2 + \frac{m^2}{c^2} } \sqrt{m^2 c^2 + \frac{n^2}{c^2} } \ge (n^2+m^2). $$
Squaring again yields
$$ \left( c^4 + \frac{1}{c^4} \right) n^2 m^2 \ge 2 n^2 m^2, $$
which holds obviously for all $c > 0$. It is easy to see that the
equality case leads to $c=1$, completing the elementary proof.

Recall that in the case of {\em hyperbolic polygons} we associated to
every billiard table a billiard surface. Let us briefly explain what
this means in our case: Reflections of the billiard table ${\bf P}_c$
along its sides leads to the rectangle $[0,2c] \times [0,2/c]$ which,
after identification of its opposite sides, becomes a torus, denoted
by $S_c$, the billiard surface associated to the billiard table
${\bf P}_c$. Then every closed billiard trajectory, traversed twice,
can be viewed as a closed geodesic in $S_c$. Such a closed geodesic is
then again, up to free homotopy, in one-one correspondence with a
vector $(2nc,2m/c)$. Our inequality above about cyclically related
closed billiard trajectories then naturally translates to a
corresponding statement about closed geodesics in the associated
billiard surfaces. The relevant Teich\-m{\"u}ller space is then the
space of all closed flat oriented surfaces (of genus $1$), which we
can identify with the hyperbolic upper half plane
${\mathbb H}^2 = \{ z \in {\mathbb C} \mid {\rm Im}(z) > 0\}$. More
concretely, we associate to every point $\tau \in {\mathbb H}^2$ the
lattice $\Gamma_{\tau}$ generated by $1$ and $\tau$, and we multiply
this lattice by a suitable mulitplicative factor, then denoted by
$\widetilde \Gamma_{\tau}$, to have covolume $4$. Then the point
$\tau \in {\mathbb H}^2$ corresponds to the marked flat surface
${\mathbb R}^2 / \widetilde \Gamma_{\tau}$. In particular, the marked
billiard surface $S_c$ corresponds to the point
$i/c^2 \in {\mathbb H}^2$, and the Weil-Petersson metric $g_{WP}$ at
$z = x + iy \in {\mathbb H}^2$ agrees, up to a multiplicative factor,
with the hyperbolic metric $\frac{dx^2+dy^2}{y^2}$ (see \cite[Section
7.3.5]{IT}). The positive vertical imaginary axis in ${\mathbb H}^2$
is therefore a Weyl-Petersson geodesic. Since this axis represents the
set of all marked billiard surfaces, we can confirm in this case that
the space of all marked billiard surfaces is a geodesically convex set
in the Teich\-m{\"uller} space ${\mathbb H}^2$.

We finish this appendix by the remark that the restriction to
\emph{Euclidean rectangles} of area one is essential: Let us consider
the bigger class of Euclidean quadrilaterals of area one (dropping the
requirement that all angles are equal to $\pi/2$). Then Figure
\ref{fig:quad-bill} illustrates that the square is no longer
necessarily the billiard table which minimises the total length of
cyclically related closed billiard trajectories: the total length of
all billiard trajectories cyclically related to the finite billiard
sequence $(1,3)$ is obviously smaller in the parallelogram. Note also
that reflections of the parallelogram along its sides does no longer
lead to a tessellation of the Euclidean plane and, therefore, we
cannot construct a billiard surface (flat torus) from this billiard
table by the above mentioned method.

\begin{figure}
  \begin{center}
    \psfrag{1}{$1$} 
    \psfrag{2}{$2$}
    \psfrag{3}{$3$} 
    \psfrag{4}{$4$}
    \includegraphics[height=5cm]{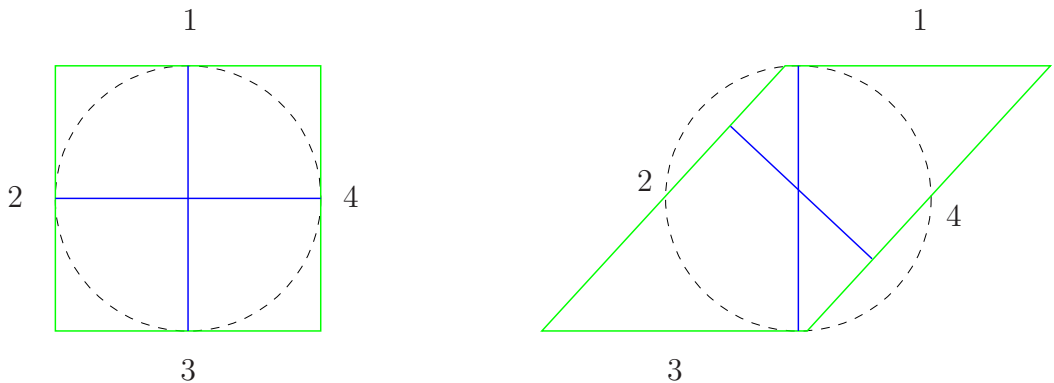}
    \caption{Closed cyclically related billiard trajectories to the
      billiard sequences $(1,3)$, $(2,4)$, $(3,1)$, and $(4,2)$. The billiard
    tables are the square and a parallelogram of area one.
   \label{fig:quad-bill}}
  \end{center}
\end{figure}


\begin{thebibliography}{99}

\bibitem{Art} E. Artin, {\em Ein mechanisches System mit
    quasiergodischen Bahnen}, Abh. Math. Sem. Univ. Hamburg {\bf 3}
  (1) (1924), 170--175.

\bibitem{BourKont} J. Bourgain and A. Kontorovich, {\em Beyond
    expansion II: Low lying fundamental geodesics}, arXiv:1406.1366,
  to appear in J. Eur. Math. Soc. (JEMS).

\bibitem{CPS} S. Castle, N. Peyerimhoff and K.F. Siburg, {\em
    Billiards in ideal hyperbolic polygons}, Discrete
  Contin. Dyn. Syst. {\bf 29} (3) (2011), 893--908.

\bibitem{Dou} F. Douma, English translation of E. Artin's article
  \cite{Art} {\em Ein mechanisches System mit quasiergodischen Bahnen}
  at \texttt{http://www.maths.dur.ac.uk/$\sim$dma0np/}

\bibitem{Hub} J. H. Hubbard, {\em Teichm\"uller Theory and
    applications to geometry, topology, and dynamics}, Vol. 1, Matrix
  Editions, Ithaca, NY, 2006.

\bibitem{IT} Y. Imayoshi and M. Taniguchi, {\em An Introduction to
    Teichm\"uller Spaces}, Springer-Verlag, Tokyo, 1992.

\bibitem{Ker83} S. P. Kerckhoff, {\em The Nielsen realization
  problem}, Ann. of Math. (2) {\bf 117} (2) (1983), 235--265.

\bibitem{Ser} C. Series, {\em The modular surface and continued fractions}, 
J. London Math. Soc. (2) {\bf 31} (1) (1985), 69--80.



\bibitem{Wol75} S. A. Wolpert, {\em Noncompleteness of the Weil-Petersson metric for Teichm\"uller space}, Pacific J. Math. {\bf 61} (2) (1975), 573--577.

\bibitem{Wol1} S. A. Wolpert, {\em On the {W}eil-{P}etersson geometry of
    the moduli space of curves}, Amer. J. Math. {\bf 107} (4) (1985), 969--997.


\bibitem{Wol87} S. A. Wolpert, {\em Geodesic length functions and the Nielsen problem}, J. Differential Geom. {\bf 25} (2) (1987), 275--296.

\end{thebibliography}
\end{document}